\documentclass[a4paper]{article}
\usepackage{amsmath}
\usepackage{amsthm}
\newtheorem{theorem}{Theorem}
\newtheorem{lemma}{Lemma}

\title{Counting Cocircuits and Convex Two-Colourings is $\#$P-complete}
\author{Andrew J. Goodall\thanks{Supported by the Heilbronn Institute for
Mathematical Research, Bristol, U.K.}\\
Department of Mathematics\\University Walk\\University of Bristol\\Bristol, BS8 1TW\\United Kingdom
\\[5mm] Steven D. Noble%
\thanks{Partially supported by the Heilbronn Institute for
Mathematical Research, Bristol, U.K.}\\Department of Mathematical Sciences\\Brunel University\\
Kingston Lane\\Uxbridge, UB8 3PH\\United Kingdom}%

\begin{document}
\maketitle
\begin{abstract}
We prove that the problem of counting the number of colourings of the vertices of a graph with at most two colours, such that the colour classes induce connected subgraphs is $\#$P-complete. We also show that the closely related problem of counting the number of cocircuits of a graph is $\#$P-complete. 
\end{abstract}

\section{Introduction}
A convex colouring of a graph $G$ is an assignment of colours to
its vertices so that for each colour $c$ the subgraph of $G$
induced by the vertices receiving colour $c$ is connected.
We consider a graph with no vertices to be connected.
The purpose
of this paper is to resolve a question of
Makowsky~\cite{mak:Question} by showing that counting the number
of convex colourings using at most two colours is $\#$P-complete.
More precisely we show that the following problem is
$\#$P-complete.
\begin{trivlist}\item[]
$\#$\textsc{Convex Two-Colourings}\\
\textbf{Input:} Graph $G$.\\
\textbf{Output:} The number of $f:V(G)\rightarrow \{0,1\}$ such
that both $G:f^{-1}(0)$ and $G:f^{-1}(1)$ are connected.
\end{trivlist}
For the definition of the complexity class $\#$P,
see~\cite{garey:comp+intract} or~\cite{papadimitriou}.

Note that the number of convex colourings using at most two
colours is equal to zero if $G$ has three or more connected
components and equal to two if $G$ has exactly two connected
components. So we may restrict our attention to connected graphs.

\section{Reductions}
All our graphs will be simple. We begin with a few definitions.
Let $X$ and $Y$ be disjoint sets of vertices of a graph $G$. The set
of edges of $G$ that have one endpoint in $X$
and the other in $Y$ is denoted by $\delta(X,Y)$. Given a connected graph $G$, a cut is a
partition of $V(G)$ into two (non-empty) sets called its
\emph{shores}. The \emph{crossing set} of a cut with shores $X$
and $Y$ is $\delta(X,Y)$. A cut is a \emph{cocircuit} if no proper
subset of its crossing set is the crossing set of a cut. Let $G$
be a connected graph. Then a cut of $G$ with crossing set $A$ is a
cocircuit if and only if the graph $G\setminus A$ obtained by removing
the edges in $A$ from $G$ has exactly two connected
components.

Note that our terminology is slightly at odds with standard usage
in the sense that the terms cut and cocircuit usually refer to
what we call the crossing set of respectively a cut and a
cocircuit. Our usage prevents some cumbersome descriptions in the
proofs. We will however abuse our notation by saying that a cut or
cocircuit has size $k$ if its crossing set has size $k$.

We consider the complexity of the following problems.
\begin{trivlist}\item[]
$\#$\textsc{Cocircuits}\\
\textbf{Input:} Simple connected graph $G$.\\
\textbf{Output:} The number of cocircuits of $G$.
\end{trivlist}

\begin{trivlist}\item[]
$\#$\textsc{Required Size Cocircuits}\\
\textbf{Input:} Simple connected graph $G$, strictly positive integer $k$.\\
\textbf{Output:} The number of cocircuits of $G$ of size $k$.
\end{trivlist}

\begin{trivlist}\item[]
$\#$\textsc{Max Cut}\\
\textbf{Input:} Simple connected graph $G$, strictly positive integer $k$.\\
\textbf{Output:} The number of cuts of $G$ of size $k$.
\end{trivlist}

\begin{trivlist}\item[]
$\#$\textsc{Monotone 2-SAT}\\
\textbf{Input:} A Boolean formula in conjuctive normal form in
which each clause contains two variables and there are no negated
literals.\\
\textbf{Output:} The number of satisfying assignments.
\end{trivlist}
It is easy to see that each of these problems is a member of
$\#$P.

The following result is from Valiant's seminal paper on
$\#$P~\cite{val:enumeration}.
\begin{theorem}~\label{th:2SAT}
$\#$\textsc{Monotone 2-SAT} is $\#P$-complete.
\end{theorem}

We will establish the following reductions.
\begin{trivlist}\item[]
\begin{center}
$\#$\textsc{Monotone 2-SAT }$\propto\#$\textsc{Max
Cut }$\propto\#$\textsc{Required Size Cocircuits}\\
$\propto\#$\textsc{Cocircuits }$\propto\#$\textsc{Convex
Two-Colourings}. \end{center}
\end{trivlist}
Combining Theorem~\ref{th:2SAT} with these reductions shows that
each of the five problems that we have discussed is
$\#$P-complete. As far as we are aware, each of these reductions
is new. We have not been able to find a reference showing that
$\#$\textsc{Max Cut} is $\#P$ complete. Perhaps it is correct to
describe this result as `folklore'. In any case our first reduction will establish this result.
Some similar problems, but not exactly what we consider
here, are shown to be $\#P$ complete in~\cite{provan:cuts}.

\begin{lemma}
$\#$\textsc{Monotone 2-SAT} $\propto \#$\textsc{Max Cut}.
\end{lemma}
\begin{proof}
Suppose we have an instance $I$ of $\#$\textsc{Monotone 2-SAT}
with variables $x_1,\ldots,x_n$ and clauses $\mathcal C=\{
C_1,\ldots,C_m\}$. We construct a corresponding instance $M(I)=(G,k)$ of
$\#$\textsc{Max Cut} by first defining a graph $G$ with vertex set
\[ \{x\} \cup \{x_1,\ldots,x_n\} \cup  \bigcup \{\{c_{i,1},\ldots,c_{i,6}\}:1\leq i \leq m\}\]
For each clause we add nine edges to $G$. Suppose $C_j$ is $x_u
\vee x_v$. Then we add the edges
\[ xc_{j,1}, c_{j,1}c_{j,2}, c_{j,2}x_u, x_uc_{j,3}, c_{j,3}c_{j,4}, c_{j,4}x_v, x_vc_{j,5}, c_{j,5}c_{j,6}, c_{j,6}x.\]
Distinct clauses correspond to pairwise edge-disjoint circuits, each
of size $9$. Now let $k =
8|\mathcal C|$. Clearly $M(I)$ may be constructed in polynomial time. We claim that the number of solutions of instance
$M(I)$ of \#{\sc Max Cut} is
equal to $2^{|\mathcal C|}$ times the number of
satisfying assignments of $I$.


Given a solution of $I$, let $L_1$ be the set of variables assigned
the value true and $L_0$ the set of variables assigned false
together with $x$. Observe that for each clause $C_j=x_u\vee x_v$
there are two choices of how to add the vertices
$c_{j,1},\ldots,c_{j,6}$ to either $L_0$ or $L_1$ so that exactly
eight edges of the circuit corresponding to $C_j$ have one
endpoint in $L_0$ and the other in $L_1$. Clearly the choices for
each clause are independent and distinct satisfying assignments
result in distinct choices of $L_0$ and $L_1$. Any of the choices
of $L_0$ and $L_1$ constructed in this way may be taken as the
shores of a cut of size $8|\mathcal C|$. Hence we have constructed
$2^{|\mathcal C|}$ solutions of $M(I)$ corresponding to each
satisfying assignment of $I$.


In any graph the intersection of a set of edges forming a circuit
and a crossing set of a cut must always have even size. So in a
solution of $M(I)$ each of the edge-disjoint circuits making up
$G$ and corresponding to clauses of $I$ must contribute exactly
eight edges to the cut. Suppose $U$ and $V\setminus U$ are the
shores of a cut of $G$ of size $8|\mathcal C|$. Then it can easily
be verified that for any clause $C=x_u \vee x_v$ both $U$ and
$V\setminus U$ must contain at least one element from
$\{x_u,x_v,x\}$. So it is straightforward to see that this
solution of $M(I)$ is one of those constructed above corresponding
to the satisfying assignment where a variable is false if and only
if the corresponding vertex is in the same set as $x$.
\end{proof}

\begin{lemma}
$\#$\textsc{Max Cut} $\propto \#$\textsc{Required Size
Cocircuits}.
\end{lemma}
\begin{proof}
Suppose $(G,k)$ is an instance of $\#$\textsc{Max Cut}. We construct an instance $(G',k')$ of
$\#$\textsc{Required Size Cocircuits} as follows. Suppose $G$ has
$n$ vertices. To form $G'$ add new vertices
$x,x',x_1,\ldots,x_{n^2}$ to $G$. Now add an edge from $x$ to
every other vertex of $G'$ except $x'$ and similarly add an edge
from $x'$ to every other vertex of $G'$ except $x$. Let
$k'=n^2+n+k$. Clearly $G'$ may be constructed in polynomial time. From each solution of the $\#$\textsc{Max Cut}
instance $(G,k)$ we construct $2^{n^2+1}$ solutions of the
$\#$\textsc{Required Size Cocircuits} instance $(G',k')$. Suppose
$C=(U,V(G)\setminus U)$ is a solution of $(G,k)$ then we may
freely choose to add $x,x',x_1,\ldots,x_{n^2}$ to either $U$ or
$V(G)\setminus U$, with the sole proviso that $x$ and $x'$ are not
both added to the same set, to obtain a cut in $G'$ of size
$k'=n^2+n+k$. Furthermore this cut is a cocircuit because both
shores contain exactly one of $x$ and $x'$ and so they induce
connected subgraphs.

Conversely suppose $C=(U,V(G')\setminus U)$ is a cocircuit in $G'$
of size $k'$. Consider the pair of edges incident with $x_j$. Note
that the partition $(x_j,V(G') \setminus x_j)$ is a cocircuit. So
if both of the edges incident with $x_j$ are in the crossing set
of $C$ then because of its minimality we must have $C=(x_j,V(G')
\setminus x_j)$ which is not possible because $C$ would then have
size $2<k'$. Now suppose that neither edge incident with $x_j$ is
in the crossing set of $C$. Then both $x$ and $x'$ lie in the same
block of the partition constituting $C$. But since $G$ is a simple
graph, the maximum possible size of such a cocircuit is at most
$2n+\binom n 2 < n^2+n+k$. Hence precisely one of the edges
adjacent to $x_j$ is in the crossing set. So $x$ and $x'$ are in
different shores of $C$. Hence the crossing set of $C$ contains:
for each $j$ precisely one edge incident to $x_j$ ($n^2$ edges in
total), for each $v \in V(G)$ precisely one of edges $vx$ and
$vx'$ ($n$ edges in total) and $k$ other edges with both endpoints
in $V(G)$. So the partition $C'=(U\cap V(G),V(G)\setminus U)$ is a
cut of $G$ of size $k$ and hence $C$ is one of the cocircuits
constructed in the first part of the proof. Consequently the
number of solutions of the instance $(G',k')$ of
$\#$\textsc{Required Size Cocircuits} is $2^{n^2+1}$ multiplied by
the number of solutions of the instance $(G,k)$ of $\#$\textsc{Max
Cut}.
\end{proof}

\begin{lemma}
$\#$\textsc{Required Size
Cocircuits }$\propto\#$\textsc{Cocircuits}
\end{lemma}
\begin{proof}
Given a graph $G$ let $N_k(G)$ denote the number of cocircuits of
size $k$ and $N(G)$ denote the total number of cocircuits. Let
$G_l$ denote the $l$-stretch of $G$, that is, the graph formed
from $G$ by replacing each edge of $G$ by a path with $l$ edges.
Let $m=|E(G)|$. Then we claim that
\[ N(G_l) =\sum_{k=1}^{m} l^kN_k(G) + \binom l 2 m.\]

To see this suppose that $C$ is a cocircuit of $G_k$. If the
crossing set of $C$ contains two edges from one of the paths
corresponding to an edge of $G$ then by the minimality of the
crossing set of $C$ we see that it contains precisely these two
edges. The number of such cocircuits is $\binom l 2 m$.

Otherwise the crossing set $C$ contains at most one edge from each
path in $G_l$ corresponding to an edge of $G$. Suppose the
crossing set of $C$ contains $k$ such edges. Let $A$ denote the
corresponding edges in $G$. Then $A$ is the crossing set of a
cocircuit in $G$ of size $k$. From each such cocircuit we can
constuct $l^k$ cocircuits of $G_l$ by choosing one edge from each
path corresponding to an edge in $A$. The claim then follows.

If we compute $N(G_1),\ldots,N(G_{m})$ then we may retrieve
$N_1(G),\ldots,N_{m}(G)$ by using Gaussian elimination because
the matrix of coefficients of the linear equations is an invertible Vandermonde matrix. The fact that the
Gaussian elimination may be carried out in polynomial time follows
from~\cite{edmonds:linalg}. \end{proof}

\begin{lemma}
$\#$\textsc{Cocircuits }$\propto\#$\textsc{Convex Colourings}.
\end{lemma}
\begin{proof}
The lemma is easily proved using the following observation. When
two colours are available, there are two convex colourings of a
connected graph using just one colour and the number of convex
colourings using both colours is equal to twice the number of
cocircuits.
\end{proof}

The preceding lemmas imply our main result.
\begin{theorem}
$\#$\textsc{Convex Colourings} is $\#$P-hard.
\end{theorem}

\section*{Acknowledgement}
We thank Graham Brightwell for useful discussions.



\begin{thebibliography}{1}

\bibitem{edmonds:linalg}
J.~Edmonds.
\newblock Systems of distinct representatives and linear algebra.
\newblock {\em Journal of Research of the National Bureau of Standards Section
  B} {\bf 71B}:241--245, 1967

\bibitem{garey:comp+intract}
M.~R. Garey and D.~S. Johnson.
\newblock {\em Computers and Intractability}.
\newblock W.~H.~Freeman, New York, 1979.

\bibitem{mak:Question}
J.~A. Makowsky.
\newblock Problem posed at the problem session at ``Building Bridges: a conference on mathematics and computer science
in honour of Laci~Lov\'{a}sz'', Budapest, 2008.

\bibitem{papadimitriou}
C.~H. Papadimitriou.
\newblock {\em Computational Complexity}.
\newblock Addison-Wesley, Reading MA, 1994.

\bibitem{provan:cuts}
J.~S. Provan and M.~O. Ball.
\newblock The complexity of counting cuts and of computing the probability that
  a graph is connected.
\newblock {\em SIAM Journal on Computing} {\bf 12}:777--788, 1983.

\bibitem{val:enumeration}
L.~G. Valiant.
\newblock The complexity of enumeration and reliability problems.
\newblock {\em SIAM Journal on Computing} {\bf 8}:410--421, 1979.

\end{thebibliography}
\end{document}